\documentclass{amsart}

\usepackage{amsaddr}
\usepackage{graphicx}
\usepackage{epstopdf}
\usepackage{dsfont}
\usepackage[usenames, dvipsnames, x11names]{xcolor}
\usepackage{hyperref}
\newtheorem{theorem}{Theorem}[section]
\newtheorem{lemma}[theorem]{Lemma}
\newtheorem{proposition}[theorem]{Proposition}
\newtheorem{corollary}[theorem]{Corollary}
\theoremstyle{definition}
\newtheorem{defi}[theorem]{Definition}
\newtheorem{example}[theorem]{Example}

\usepackage{tikz}
\theoremstyle{remark} \newtheorem{remark}[theorem]{Remark}

\definecolor{jaune}{HTML}{FCBB6D}
\definecolor{rouge}{HTML}{AB6C82}
\definecolor{noir}{HTML}{475C7A}

\numberwithin{equation}{section}

\newcommand{\abs}[1]{\lvert#1\rvert}

\newcommand{\fonction}[5]{#1 :\left\{\begin{array}{rcl} #2 &
\longrightarrow & #3\\ #4 & \longmapsto & #5\\
\end{array} \right.}
\newcommand{\ma}[1]{\mathbb #1}
 
\newcommand{\mc}[1]{\mathcal #1}

\begin{document}

\title{Flat traces for a random partially hyperbolic map} %

\date\today \author{Luc Gossart} \address{Institut Fourier, 100, rue
des maths BP74 38402 Saint-Martin d'Heres France\footnote{2010
Mathematics Subject Classification.  37D30 Partially hyperbolic
systems and dominated splittings, 37E10 Maps of the circle, 60F05
Central limit and other weak theorems, 37C30 Zeta functions,
(Ruelle-Frobenius) transfer operators, and other functional analytic
techniques in dynamical systems.  }}

\maketitle
\begin{abstract} We consider a $\mathbb R/\mathbb Z$ extension of an
Anosov diffemorphism of a compact Riemannian manifold by a random
function $\tau$ and show that the flat traces of the transfer
operator, reduced with respect to frequency in the fibers, converge in
law towards Gaussians, up to an Ehrenfest time that decreases with the
regularity of $\tau$.
\end{abstract}

\newpage
\tableofcontents

\newpage

\section{Introduction}

This paper follows \cite{gossart2020flat} and extends its results to
the case of an Anosov diffeomorphism on a compact Riemannian manifold.

The main object of study are the flat traces of the transfer operator.
This operator acts by pulling back functions and its spectral
properties are linked to the dynamical correlations.  For Anosov
diffeomorphisms, the statistical properties have been studied since
the late 1960's, with help of Markov partitions and symbolic dynamics
\cite{bowen1975ergodic}.  The construction of spaces, in the Anosov
framework, in which the transfer operator is quasicompact was first
achieved later by Blank, Keller and Liverani \cite{blank2002ruelle}.

For Anosov flows or partially hyperbolic diffeomorphisms, the neutral
direction adds a substantial difficulty to the study.  In this
setting, Dolgopyat \cite{dolgopyat1998decay} showed exponential decay
of correlations for the geodesic flow on negatively curved surfaces,
and Liverani \cite{liverani2004contact} generalized this result to all
$\mc C^4$ contact Anosov flows, by constructing anisotropic Banach
spaces in which the generating vector field has a spectral gap and
resolvent.  Tsujii \cite{tsujii2010quasi} extended this method and
showed quasi-compactness of the transfer operator itself, with an
explicit bound on the essential spectral radius, for contact Anosov
flows, in some Hilbert spaces.  Butterley and Liverani
\cite{butterley2007smooth} then constructed Banach spaces to study the
spectrum of general Anosov flows.  Weich and Bonthonneau
\cite{bonthonneau2017ruelle} on their side constructed outside the
scope of compact manifolds appropriate spaces for geodesic flow on
negatively curved manifolds with a finite number of cusps.  Dyatlov
and Guillarmou \cite{dyatlov2016pollicott} did it for open hyperbolic
systems.

A simple example of Anosov flow is the suspension of an Anosov
diffeomorphism, or the suspension semi-flow of an expanding map.
Pollicott \cite{pollicott1985rate} showed exponential decay of
correlations in this setting and Tsujii constructed suitable
spaces for the transfer operator and gave an upper bound on its
essential spectral radius in \cite{tsujii2008decay}.

Here we consider a close model, namely a $\mathbb R$-extension of an
Anosov diffeomorphism on a compact Riemannian manifold, for which
dolgopyat \cite{dolgopyat2002mixing} has shown generic rapid decay of
correlations.  In a series of papers, de Simoi, Liverani, Poquet and
Volk \cite{de2017fast} and de Simoi and Liverani
\cite{de2016statistical} \cite{de2018limit} studied statistical
properties of fast-slow dynamical systems. Their model generalizes
$\ma T$ extensions of circle expanding maps.

We investigate a small random perturbation of the roof function and
show that the flat traces (\ref{trace plate}) of the iterates of the
transfer operator (restricted to a given frequency $\xi$ in the fiber
direction) satisfy a central limit theorem in a semiclassical regime
linking time $n$ and frequency $\xi$ (theorem \ref{theorem main}).  We
obtain convergence towards a Gaussian law up to a constant times the
Ehrenfest time, this constant being a decreasing function of the
regularity of the random function.  The principle is the same as in
\cite{gossart2020flat}: We show pointwise convergence of the
characteristic function, by decomposing for each time $n$ the roof
function as the sum of a random function that decorrelates at a scale
corresponding to the minimal distance between periodic points of
period $n$, and an other function that plays no role if the frequency
is large enough.

Naud \cite{naud2016rate} found in the case of circle extensions of
some analytic Anosov maps of the torus lower bounds on the first
eigenvalue, both in the deterministic and random settings.  He makes
use of the fact that with positive probability, there is a lower bound
on the modulus of the trace (of same order as the scaling $A_n$ from
(\ref{A_n})), and takes advantage of the fact that the operator is
trace class.  This is not the case in our setting and we don't
know whether information on the Ruelle-Pollicott spectrum can be
recovered from our estimation of the flat traces.

\subsection{Model}\label{sec:model}
Let $M$ be a smooth Riemannian manifold of dimension $d$ and
$T:M\longrightarrow M$ be a transitive Anosov diffeomorphism.  This
means that the tangent bundle admits a splitting $TM=E^u\oplus E^s$
such that\vspace{0.5cm}
\begin{center}
\begin{tabular}{cc} $(1)$ & $ dT_x\left(E^i(x)\right)= E^i(f(x)),\ i
\in\{u,s\},$\vspace{0.2cm}\\ $(2)$ & $ \exists 0<\lambda<1, \exists
C>0,\forall n\in\ma N, \left|\begin{array}{lc} \forall v\in E^u,
\left\|dT^{-n}\cdot v\right\|&\leq C \lambda^n\|v\|,\\ \forall v\in
E^s,\ \left\|dT^{n}\cdot v\right\|&\leq
C\lambda^n\|v\| \end{array}\right.$
\end{tabular}
\end{center} \vspace{0.3cm} and that $T$ has a dense orbit.  We will
be interested, given $k\geq0$ and a $\mc C^k$ function $\tau$, in the
skew-product

\begin{equation} \fonction F {M\times\ma R} {M\times\ma
R}{(x,y)}{\left(T(x),y+\tau(x)\right)}.
\end{equation}
\subsection{Ruelle spectrum} To the map $F$ can be associated a
transfer operator $\mc L_\tau$ acting on $\mc C^k(M\times\ma R)$ by
composition:
\begin{equation} \mc L_\tau v = v\circ F.
\end{equation}
Fourier analysis with respect to $y$ leads to the introduction of the
family of operators on $\mc C^k(M)$ indexed by $\xi\in\ma R$

\begin{equation} \mc L_{\xi,\tau}u=e^{i\xi\tau}u\circ T.
\end{equation}
Indeed, if $v$ is a Fourier mode with respect to $y$, that is
$v(x,y)=u(x)e^{i\xi y}$ for some $u\in\mc C^k(M),\in\ma R$,

\begin{equation} \mc L_\tau v (x,y)=\mc L_{\xi,\tau}u(x) e^{i\xi y}.
\end{equation}
These operators can be extended to distributions by duality.  They
have their essential spectral radius bounded by explicit constants in
appropriate spaces.  The operators are not trace class, but we can
define a generalization of their trace, called flat trace, which has a
connexion with their spectrum.  See Appendix \ref{annexe spectre de
Ruelle} for a brief discussion about this.  The flat trace is the main
object studied in this paper, we express a central limit theorem for a
small random perturbation of a given function $\tau$ in the limit of
large times $n$ and frequencies $\xi$ in Theorem \ref{theorem main}.
Its expression involves periodic points and is given by

\begin{equation}\label{trace plate} \mathrm{Tr}^\flat(\mc
L^n_{\xi,\tau})=
\sum_{x,T^n(x)=x}\frac{e^{i\xi\tau_x^n}}{\left|\det(1-d(T^n)_x)\right|},
\end{equation}
where $\tau^n_x$ denotes the Birkhoff sum
\begin{equation}\label{somme de Birkhoff} \tau^n_x:=
\tau(x)+\tau(T(x))+\cdots+\tau(T^{n-1}(x)).
\end{equation}
\subsection{Eigenfunction Gaussian random fields}
\begin{defi} We will call centered Gaussian field on $M$ a random
distribution of the form
\begin{equation} f = \sum_{j\geq 0}c_j \zeta_j \phi_j,
\end{equation}
where the $c_j\geq0$ grow at most polynomially with
$j$, $\zeta_j$ are i.i.d centered Gaussian random variables of
variance 1 and $(\phi_j)_j$ is a Hilbert basis of eigenfuntions of the
Laplace-Beltrami operator:

\begin{equation} \Delta\phi_j=\lambda_j\phi_j,
\end{equation}
$0=\lambda_0\leq\lambda_1\leq\lambda_2\leq\cdots$.
\end{defi} This sum is in general understood in the sense of
distributions, Proposition \ref{proposition regularite champs
gaussiens} thereafter expresses a link between the growth of
$(c_j)_{j\geq0}$ and the regularity of the field.  We will only be
interested in at least continuous fields in what follows.

\begin{example} If $c_j=1$ for all $j$, the random field $W$ is called
white noise:
\begin{equation}
    \label{bruit blanc} W=\sum_{j\geq0}\zeta_j\phi_j.
  \end{equation}
It is a random distribution, almost surely not in
$L^2(M).$
\end{example}

On $M$, we have a notion of Sobolev spaces:
\begin{defi} Let $s\in\ma R$.  The Sobolev space $H^s(M)$ is defined
by

\begin{equation} H^s(M)=\left(1+\Delta\right)^{-s/2}L^2(M).
\end{equation}

\end{defi}
\begin{proposition}\label{proposition regularite champs gaussiens}
Assume that

\begin{equation} c_j=O(j^{-\alpha})
\end{equation}
for some $\alpha\in\ma R$.  Then, almost surely, the centered Gaussian
field

\begin{equation} f:=\sum_{j\geq 0}c_j\zeta_j\phi_j\in H^s(M)
\end{equation}
for every $s<d(\alpha-\frac12)$.  Thus,
\begin{equation} f\in\mc C^k(M)
\end{equation} for every $k<d(\alpha-1)$ (where $\mc C^k(M)$ is
understood as $(\mc C^{-k}(M))'$ for negative $k$.)
\end{proposition}
\begin{proof} See appendix \ref{annexe champs gaussiens}.
\end{proof}

\subsection{Result} If $x$ is a periodic point of $T$, we write its
primitive period $m_x$.  Let us define the amplitudes $A_n$ by

\begin{equation}\label{A_n}
A_n:=\left(\sum_{T^n(x)=x}\frac{m_x}{\left|\det(1-dT^n_x)\right|^2}\right)^{-\frac12}.
\end{equation}
Let also \begin{equation} \Lambda^{\pm}:=\lim_{n\to\infty}\max_{x\in
M}\|dT_x^{\pm n}\|^{\frac1n}
\end{equation} and

\begin{equation}\label{definition coefficient expansion Lambda}
\Lambda:=\max(\Lambda^{\pm}).
\end{equation}

Let $h_{\mathrm{top}}\underset{(\ref{equation h_top leq log
Lambda})}\leq\frac d2\log\Lambda$ be the topological entropy of the
map $T$ (see Definition 3.1.3 in \cite{katok1997introduction}).

\begin{theorem}\label{theorem main} Let us fix any $\tau_0\in \mc
C^0(M)$ and $\varepsilon>0$.  Let

\begin{equation} \delta\tau=\sum_{j\geq 0}c_j\zeta_j\phi_j
\end{equation}
be a centered Gaussian field with $\zeta_j$ i.i.d.
$\mc N(0,1)$ such that
\begin{equation}\label{condition c_n}
\frac1Cj^{-\alpha}\leq c_j\leq C j^{-\beta}
\end{equation}
for some constants $C>0,\alpha>\beta>1.$ By Proposition
\ref{proposition regularite champs gaussiens}, the condition involving
$\beta$ ensures that $\delta\tau$ is almost surely continuous, and
that we can define the flat trace of $\mathcal L^n_{\xi,\tau}$ for
$\tau:=\tau_0+\varepsilon\delta\tau$.  Then, for any $0<c<1$, we have
the following convergence in law
\begin{equation}\label{convergence trace plate}
A_n\mathrm{Tr}^\flat\left(\mc L^n_{\xi,\tau}\right)\longrightarrow\mc
N_{\ma C}(0,1)
\end{equation}
as $n$ and $\xi$ go to infinity under the relation
\begin{equation}\label{relation n xi} n\leq
c\frac{\log\xi}{h_{\mathrm{top}}+\frac{d}{2}(\alpha-\frac12)\log
\Lambda}.
\end{equation}
 
\end{theorem}
\begin{remark} $\beta$ plays no other role than to make sure that the
flat trace is well defined.
\end{remark}
\begin{remark} If for some reason we want to impose a certain
regularity on the function $\delta\tau$ using Proposition
\ref{proposition regularite champs gaussiens}, we need to take
$\alpha$ large enough, and the larger it is, the more restrictive
condition \ref{relation n xi} imposed on the time $n$ is.
\end{remark}
\begin{remark} In \cite{gossart2020flat}, we obtained instead of
(\ref{relation n xi}) the condition

\begin{equation} n\leq c\frac{\log\xi}{\log l +(k+\frac
12+\frac{\delta}{2})\log M}.
\end{equation}

In this setting, $k+\frac12+\frac\delta2$ was the analog of $\alpha$,
and $\log l$ the topological entropy.  $M$ was analogous to
$\Lambda^{\frac d2}$.  The dimension was $1$ and the exponent
$\frac 12$ comes from the fact that $T$ is invertible, while $E$ was
not and had therefore possibly denser periodic points.  $M=\sup E'$
could have been refined as $\lim_n(\sup(E')^n)^{\frac1n}$ to match the
definition of $\Lambda$.  With this in mind, the bound of
\cite{gossart2020flat} translates to

\begin{equation} n\leq c\frac{\log\xi}{h_{\mathrm{top}}
+\frac d2\alpha\log \Lambda}.
\end{equation}

We obtain here a slightly better bound, with $\alpha-\frac12$ instead
of $\alpha$, due to the fact that our proof directly deals with the
multivariate Gaussian probability density in a space of dimension
approximately $e^{nh_{\mathrm{top}}}$ instead of reducing to the
unidimensional variables.
\end{remark}

\section{Proof}

\begin{remark}\label{remarque abus de notation} Let us first remark
that the convergence in law stated in theorem \ref{theorem main} only
involves the law of $\delta\tau$.  Therefore, we will abusively name
$\delta\tau$ another field that has the same law.
\end{remark}

\subsection{Sketch of proof} Let us choose $\alpha>1 $, $\eta>0$ and
$\delta\tau$ as in the statement of Theorem \ref{theorem main}.  We
will show that condition (\ref{condition c_n}) allows us to construct
a centered Gaussian field with the same law as $\delta\tau$, as a sum
of independent Gaussian fields

\begin{equation} \delta\tau=\delta\tau_0+\sum_{j\geq 1}\delta\tau_j,
\end{equation} so that the covariances $\ma
E[\delta\tau_j(x)\delta\tau_j(y)]$ become very small at a distance
greater than $\Lambda'^{-\frac j2}$, for some $\Lambda'>\Lambda$ that
we will choose small enough (\textit{i.e.} close enough to $\Lambda$),
which is smaller than the minimal distance between two periodic points
of large period $j$, as we know from Lemma \ref{lemme distance points
periodiques}.  Here we have used the abusive notation described in the
preliminary remark \ref{remarque abus de notation}.  Therefore, the
phases appearing in the trace formula (\ref{trace plate}) will behave
as independent random variables on $S^1$, almost uniform when $\xi$ is
large enough, that is, under the condition (\ref{relation n xi}).

\subsection{Construction of the fields $\delta\tau_j$}

Let $\chi\in\mc S(\ma R)$ be a positive Schwartz function such that
$\chi(0)=0$, normalized so that
\begin{equation}\label{eq normalisation chi} \int\chi^2\ \mathrm dx=(2\pi)^d.
\end{equation}
Let us define the family of operators

\begin{equation}
    \label{defi Ph} P_h:=\chi(h^2\Delta).
\end{equation}

\begin{proposition}[\cite{zworski2012semiclassical} Theorem 14.9 p.358
and Theorem 9.6 p.209]\label{prop decroissance noyau} $P_h$ is a
$h$-pseudodifferential operator and its Schwartz kernel $K_h$
satisfies
\begin{equation}
    \label{estimee noyau} \forall N>0,\exists C_N>0,\forall x\neq
y,\left| K_h(x,y)\right|\leq \frac{C_N h^N}{d(x,y)^N}.
\end{equation}
\end{proposition} Let $W_j$ be a family of independent white noises
(defined in (\ref{bruit blanc})) independent of $\delta\tau$.  Let
$\gamma>0$ and $\widetilde\Lambda>\Lambda$ to be chosen small enough
later and
\begin{equation}\label{definition h_j}h_j=\widetilde\Lambda^{-\frac
j2}\end{equation}
  
$P_{h_j}$ is a positive selfadjoint operator, so we can define
\begin{equation}
    \label{defi delta tau j}
\delta\tau_j:=h_j^{d\alpha+\gamma}\sqrt{P_{h_j}}W_j.
\end{equation}
In other terms, there exist i.i.d.  random variables $\zeta_{j,k}$ of
law $\mc N(0,1)$ such that

\begin{equation}
    \label{delta tau j} \delta\tau_j=h_j^{d\alpha+\gamma}\sum_{k\geq
1} \sqrt{\chi(h_j^2 \lambda_k)}\zeta_{j,k}\phi_k.
\end{equation}
\begin{lemma} Let $K_j$ be the Schwartz kernel of the operator
$h_j^{2(d\alpha+\gamma)}P_{h_j}$.  Then,
\begin{equation} \ma E[\delta\tau_j(x)\delta\tau_j(y)]=K_j(x,y).
\end{equation}

Moreover, on the diagonal, \cite[Theorem 14.10
p.361]{zworski2012semiclassical}
\begin{equation}
    \label{valeur noyau diagonale} K_j(x,x)\underset{(\ref{eq
normalisation chi})}{=}h_j^{d(2\alpha-1)+2\gamma}(1+O(h_j)).
\end{equation}
\end{lemma}

\begin{proof} Indeed, on one hand

\begin{equation}
\begin{split} \ma E[\delta\tau_j(x)\delta\tau_j(y)]&=
h_j^{2(d\alpha+\gamma)} \sum_{k,k'\geq0}\sqrt{\chi(h_j^2
\lambda_k)}\sqrt{\chi(h_j^2 \lambda_k')}\ma
E[\zeta_{j,k}\zeta_{j,k'}]\phi_k(x)\phi_{k'}(y)\\
&=h_j^{2(d\alpha+\gamma)}\sum_{k\geq0}\chi(h_j^2
\lambda_k)\phi_k(x)\phi_{k}(y),
\end{split}
\end{equation}
and on the other hand, since
\begin{equation} P_{h_j}\phi_k=\chi(h_j^2\lambda_k)\phi_k,
\end{equation}
we have for any $u,v\in L^2(M)$
\begin{equation}
\begin{split} \left\langle v,P_{h_j}u\right\rangle_{L^2(M)}&=\int \bar
v(x)\sum_{k\geq 0}\left(P_{h_j}\phi_k\right)(x)\int u(y)\phi_k(y)\
\mathrm{d}y\mathrm{d}x\\ &=\int \left(\sum_{k\geq
0}\chi(h_j^2\lambda_k)\phi_k(x)\phi_k(y)\right)\bar v(x)u(y)\ \mathrm
dx\mathrm dy.
    \end{split}
\end{equation}
\end{proof}

Let us choose $\Lambda<\Lambda'<\widetilde\Lambda$ (recall that
$\widetilde\Lambda$ is involved in the definition of $h_j$ in
(\ref{definition h_j})).  As a consequence of Lemma \ref{lemme
distance points periodiques}, Proposition \ref{prop decroissance
noyau}, and the definition of $h_j$ we have the following decay:

\begin{equation}
    \label{k(x,y)}
\begin{split} \forall N>0,\exists C_N>0,\forall x\neq y\in M,\forall
j\text{ large enough,}\\ (T^jx=x \text{ and }
T^jy=y)\implies|K_j(x,y)|&\leq C_N
\frac{h_j^{N+2(d\alpha+\gamma)}}{d(x,y)^N}\\ &\leq C_N
\frac{h_j^{N}}{d(x,y)^N}\\ &\leq
\frac{C_N}{C^N}\left(\frac{\Lambda'}{\widetilde\Lambda}\right)^{\frac{N}2j}
\end{split}\end{equation} for some $C>0$.  Thus,

\begin{equation}
    \label{decroissance de la covariance pour les points periodiques}
\forall k>0,\exists C_k>0,\forall x\neq y\in M, (T^jx=x \text{ and }
T^jy=y)\implies|K_j(x,y)|\leq C_k e^{-k j}.
\end{equation}
\begin{lemma} Let us write as in Theorem \ref{theorem main}

\begin{equation} \delta\tau=\sum_{k\geq0}c_k\zeta_k\phi_k.
\end{equation}

We have
\begin{equation} \sum_{j\geq1}\delta\tau_j = \sum_{k\geq
0}c_{k}'\zeta_k'\phi_k,
\end{equation} where $(\zeta_k')_k$ is a family of i.i.d.  random
variables of law $\mc N(0,1)$ independent of the variables $\zeta_k$,
and

\begin{equation} c_{k}'=O\left(c_k\right).
\end{equation}

\end{lemma}
\begin{proof}

\begin{equation} \sum_j\delta\tau_j\underset{(\ref{defi delta tau
j})}=\sum_{k\geq
1}\left(\sum_{j\geq1}h_j^{d\alpha+\gamma}\sqrt{\chi(h_j^2\lambda_k)}\zeta_{j,k}\right)\phi_k.
\end{equation}

Since every variables are independent from each other, the

\begin{equation}
\left(\sum_{j\geq1}h_j^{d\alpha+\gamma}\sqrt{\chi(h_j^2\lambda_k)}\zeta_{j,k}\right)_k
\end{equation}
are independent Gaussian variables of variances

\begin{equation}
{c_k'}^2:=\sum_{j\geq1}h_j^{2d\alpha+2\gamma}\chi(h_j^2\lambda_k).
\end{equation}

Now, since $\chi\in\mc S(\ma R)$,

\begin{equation} \exists C>0,\forall j,k,\chi(h_j^2\lambda_k)\leq\frac
C{h_j^{2d\alpha}\lambda_k^{{d\alpha}}}.
\end{equation}

Also,
\begin{equation} \sum_{j\geq1}h_j^{2\gamma}<\infty
\end{equation}
so the variances satisfy
\begin{equation} {c_k'}^2=\sum_{j\geq1}h_j^{2
d\alpha+2\gamma}\chi(h_j^2\lambda_k)=O(\lambda_k^{- {d\alpha}})
\end{equation}

By Weyl's law \cite[(14.3.21) p.362]{zworski2012semiclassical} there
exists a constant $C$ depending on $M$ such that

\begin{equation}\label{eq:Weyl law} \lambda_k^d\sim C k^2
\end{equation}
so
\begin{equation} {c_k'}^2=O(k^{-2\alpha}).
\end{equation}

\end{proof}

\begin{corollary}\label{corollaire delta tau 0} Consequently, up to
the multiplication of each $\delta\tau_j$ by the same constant,
condition (\ref{condition c_n}) allows us to define the field

\begin{equation}
\delta\tau_0:=\sum_{k\geq1}\sqrt{c_k^2-{c_k'}^2}\zeta_k''\phi_k,
\end{equation} for i.i.d.  random variables $\zeta_k''$ of law $\mc
N(0,1)$ independent of the variables $\zeta_k$ and $\zeta_k'$.
According to Remark \ref{remarque abus de notation}, we will abusively
write
\begin{equation}\label{delta tau=somme delta tau j} \delta\tau =
\sum_{j\geq 0}\delta\tau_j.
\end{equation}
\end{corollary}

\begin{defi} We will use the following notations for periodic orbits
in this paper: $\mathrm{Per}(n)$ will be the set of periodic orbits of
period $n$, while $\mc P_m$ will be the set of periodic orbits of
primitive period $m$.  This way, we have a disjoint union

\begin{equation} \mathrm{Per}(n)=\coprod_{m|n}\mc P_m.
\end{equation}

\end{defi}

If $O\in\mathrm{Per}(n),$ then the Birkhoff sums
$f^n_x=\sum_{k=0}^{n-1}f(T^kx)$ do not depend on the point $x\in O$
and will be written $f^n_O$.  Similarly, $\det (1-dT^n_O)$ will denote
the Jacobian $\det (1-d(T^n)_x)$ for any $x\in O.$

\begin{proposition}[{\cite[Theorem 18.5.5
p.585]{katok1997introduction} }]
\label{corollaire card per(n)} There exists $C>0$ such that
\begin{equation}\label{card per(n)} \frac1Ce^{nh_{\mathrm{top}}} \leq
\#\mathrm{Per}(n)\leq Ce^{nh_{\mathrm{top}}}.
  \end{equation}

\end{proposition}

\begin{remark} Note that, since we assume the map $T$ to be
transitive, by the Closing Lemma \cite[Theorem 6.4.15
p.269]{katok1997introduction} it has periodic orbits of arbitrary
large period.  Therefore, necessarily, $h_{\mathrm{top}}>0$.
Thus, \begin{equation}
\#\mathrm{Per}(n)\underset{n\to\infty}{\longrightarrow}\infty.
\end{equation}
\end{remark}

\begin{remark} Let us notice that Lemma \ref{lemme distance points
    periodiques} implies
  \begin{equation}\label{equation h_top leq log
Lambda} h_{\mathrm{top}}\leq\frac d2\log\Lambda.
\end{equation}
\end{remark}
\begin{proof} Let indeed $\Lambda'>\Lambda$.  $M$ can be covered by
$O(\Lambda'^{\frac{nd}2})$ balls of radius $\Lambda'^{-\frac n2}$.
The constraint of Lemma \ref{lemme distance points periodiques}
implies that each ball of radius $\Lambda'^{-\frac n2}$ contains a
bounded number of points of $\mathrm{Per}(n)$.  (\ref{card per(n)})
then implies

\begin{equation} e^{nh_{\mathrm{top}}}\leq C\Lambda'^{\frac {nd}2}
\end{equation} for every $\Lambda'>\Lambda.$
\end{proof}

Recall that
\begin{equation}\label{definition delta tau}
\tau=\tau_0+\varepsilon\delta\tau\underset{(\ref{delta tau=somme delta
tau j})}=\tau_0+\varepsilon\sum_{j\geq0}\delta\tau_j.
\end{equation}

\begin{lemma}\label{lemme X^n_O Y^n_O} We can write

\begin{equation} \mathrm{Tr}^\flat(\mc
L^n_{\xi,\tau})=\sum_{m|n}m\sum_{O\in\mc
P_m}\frac{e^{i\varepsilon\xi(X^n_O+Y^n_O)}}{\left|\det\left(1-dT^n_O\right)\right|},
\end{equation}
where \begin{enumerate}
    \item for all $n\in\ma N$, $\left(
X^n_O\right)_{O\in\mathrm{Per}(n)}$ is a Gaussian random vector such
that
\begin{equation}\label{matrice de covariance C_n diag} \exists
C>0,\forall n\in\ma N,\forall O\in\mathrm{Per}(n),\ \frac 1 Cn
h_n^{d(2\alpha-1)+2\gamma}\leq \sigma_O^2:=\ma E[(X^n_O)^2]\leq
Cn^2h_n^{d(2\alpha-1)+2\gamma}
\end{equation} and
\begin{equation}\label{matrice de covariance C_n hors diag}
\forall\alpha>0,\exists C_\alpha,\forall n\in\ma N,\forall O\neq
O'\in\mathrm{Per}(n),\ |\ma E[X^n_OX^n_{O'}]|\leq C_\alpha e^{-\alpha
n}.
\end{equation}
\item For every integer $n$ the random variable
$(Y^n_O)_{O\in\mathrm{Per}(n)}$ is independent of
$\left(X^n_O\right)_{O\in\mathrm{Per}(n)}$.
\end{enumerate}
\end{lemma}

\begin{proof}[Proof of lemma \ref{lemme X^n_O Y^n_O}] Using
~(\ref{trace plate}), and the fact that the Birkhoff sums and
differentials $dT^n_x$ only depend on the orbit, one can pack the
terms
\begin{equation}
    \begin{split} \mathrm{Tr}^\flat(\mc
L^n_{\xi,\tau})&=\sum_{m|n}m\sum_{O\in\mc
P_m}\frac{e^{i\xi\tau^n_O}}{|\det(1-dT^n_O)|}
\end{split}
\end{equation}

Now, we isolate the term $\delta\tau_n$ in (\ref{definition delta
tau}) (where $n$ is the time appearing in the expression
$\mathrm{Tr}^\flat(\mc L^n_{\xi,\tau})$) and set

\begin{equation} X^n_O:=(\delta\tau_n)^n_O
\end{equation}
and
\begin{equation} Y^n_O:=\left(\frac{\tau_0}{\varepsilon}+\sum_{j\neq
n}\tau_j\right)^n_O.
\end{equation}
The independence of the family
$(\zeta_{j,n})_{\substack{n\geq 0 \\j\geq1}}\cup (\zeta''_k)_{k\geq0}$
involved in (\ref{delta tau j}) and Corollary \ref{corollaire delta
tau 0} gives the independence between the families $(X^n_O)_O$ and
$(Y^n_O)_O$.  Then, for an orbit $O\in\mc P_m\subset\mathrm{Per}(n),$

\begin{equation}
\begin{split} X^n_O&=(\delta\tau_n)^n_O\\ &=\frac nm\sum_{x\in
O}\delta\tau_n(x).
         \end{split}
\end{equation}

Thus $X^n_O$ is a Gaussian random variable and

\begin{equation} (X^n_O)^2=\frac{n^2}{m^2}\left(\sum_{x\in
O}\delta\tau_n(x)^2+\sum_{\substack{x,y\in O\\x\neq
y}}\delta\tau_n(x)\delta\tau_n(y)\right).
\end{equation}
So
\begin{equation}
\begin{split} \ma E[(X^n_O)^2]&=\frac{n^2}{m^2}\left(\sum_{x\in
O}K_n(x,x)+\sum_{\substack{x,y\in O\\x\neq y}}K_n(x,y)\right)\\
&\underset{(\ref{valeur noyau diagonale}),(\ref{decroissance de la
covariance pour les points
periodiques})}{=}\frac{n^2}{m}h_n^{d(2\alpha-1)+2\gamma}(1+O(h_n))+O(e^{-\alpha
n})\end{split}
\end{equation}
for every $\alpha>0.$ Similarly
\begin{equation}
\begin{split} \ma E[X^n_OX^n_{O'}]&=\ma E\left[\sum_{x\in
O}\delta\tau_n(x)\sum_{y\in O'}\delta\tau_n(y)\right]\\ &=\sum_{x\in
O,y\in O'}K_n(x,y)\\ &=O(e^{-\alpha n})
\end{split}
\end{equation}
for every $\alpha>0.$

 This gives the expressions (\ref{matrice de covariance C_n diag}) and
(\ref{matrice de covariance C_n hors diag}).
\end{proof}

By Levy's theorem, in order to get the convergence in law of
$A_n\mathrm{Tr}^\flat(\mc L^n_{\xi,\tau})$ towards $\mc N_{\ma
C}(0,1)$, it is sufficient to show that the characteristic function

\begin{equation} \ma
E\left[e^{i\left\langle(\mu,\nu),A_n\mathrm{Tr}^\flat(\mc
L^n_{\xi,\tau})\right\rangle_{\ma R^2}}\right]
\end{equation}
converges pointwise towards
$e^{-\frac{\mu^2+\nu^2}{4}}$.

\begin{proposition}\label{prop emunu} The characteristic function of
the rescaled flat traces

\begin{equation} E(\mu,\nu):=\ma E \left[\exp\left(
{i\left(\mu\mathrm{Re}\left(A_n\mathrm{Tr}^\flat(\mc
L^n_{\xi,\tau})\right)+\nu\mathrm{Im}\left(A_n\mathrm{Tr}^\flat(\mc
L^n_{\xi,\tau})\right)\right)} \right)\right]
\end{equation}
satisfies

\begin{equation} E(\mu,\nu)\sim\prod_{m|n}\prod_{O\in\mc P_m}\int
\exp\left( {i\frac{mA_n}{\abs{\det(1-dT^n_O)}}(\mu\cos x+\nu\sin x)}
\right)\mathrm dx
\end{equation}
under condition (\ref{relation n xi}).
\end{proposition}
\begin{proof} Let $n\in\ma N$ and $m|n$.  Let us write
  \begin{equation}
    N:=\#\mathrm{Per}(n)\underset{(\ref{card
per(n)})}{\asymp}e^{nh_{\mathrm{top}}}\underset{(\ref{equation h_top
leq log Lambda})}=O\left(\Lambda^{\frac {n d}2}\right).
\end{equation}
Let us write for $O\in\mc P_m$ and $t\in\ma R$
\begin{equation} f_O(t):=\exp\left(
{i\frac{mA_n}{\abs{\det(1-dT^n_O)}}(\mu\cos(t)+\nu\sin(t))} \right)
\end{equation}
and for $x\in\ma R^N$
\begin{equation} g_n(x):=\frac{e^{-\frac12\left\langle
x,\Sigma_n^{-1}x\right\rangle}}{\sqrt{(2\pi)^N\det(\Sigma_n)}},
\end{equation}
where $\Sigma_n$ is the covariance matrix of
$(X^n_O)_{O\in\mathrm{Per}(n)}$: From Lemma \ref{lemme X^n_O Y^n_O},

\begin{equation}\label{eq sigma n}
\Sigma_n=\mathrm{Diag}(\sigma_O^2)_{O\in\mathrm{Per}(n)}+R_n,
\end{equation}
$R_n$ being a matrix with entries uniformly
$O(e^{-kn})$ for every $k>0$.  The proposition will quickly lead to
the pointwise convergence of the characteristic function, and its
proof is a consequence of the following technical lemmas:

\begin{lemma}\label{lemme majoration E(mu,nu)} For $k\in\ma Z^{N}$,
let
\begin{equation} H_k:=\prod_{O\in\mathrm
{Per}(n)}\left[\frac{2\pi}{\xi}k_O,\frac{2\pi}{\xi}(k_O+1)\right].
\end{equation}
\begin{equation}
\left|\frac{E(\mu,\nu)}{\prod_{O\in\mathrm{Per}(n)}\int_0^{2\pi}f_O(t)\mathrm
dt}-1\right|\leq4\pi\frac{\sqrt N}{\xi}\sum_{k\in\ma
Z^N}\left(\frac{2\pi}{\xi}\right)^N\sup_{H_k}\|\nabla g_n\|
\end{equation}

\end{lemma}

\begin{lemma}\label{lemme accroissements finis} $\frac{\sqrt
N}{\xi}\sum\limits_{k\in\ma
Z^N}\left(\frac{2\pi}{\xi}\right)^N\sup\limits_{H_k}\|\nabla
g_n\|\longrightarrow0$ under condition (\ref{relation n xi}).
\end{lemma}
\begin{proof}[Proof of Lemma \ref{lemme majoration E(mu,nu)}] Using
Lemma \ref{lemme X^n_O Y^n_O}, we see that
\begin{equation} E(\mu,\nu)=\int_{\ma R^{\mathrm{Per}(n)}}
\left(\int_{\ma R^{\mathrm{Per}(n)}} \left(\prod_{O\in\mathrm{Per}(n)}
f_O(\xi(x_O+y_O))\right)g_n(x)\mathrm dx\right)d\ma P_Y(y)
\end{equation}
Let us write for a given $y=(y_O)_{O\in\mathrm{Per}(n)}$
    \begin{equation}\label{E_y} E_y(\mu,\nu)= \int_{\ma
R^{\mathrm{Per}(n)}} \left(\prod_{O\in\mathrm{Per}(n)}
f_O(\xi(x_O+y_O))\right)g_n(x)\mathrm dx
\end{equation}
Since, for fixed $y$ the functions $x_O\mapsto f_O(\xi(x_O+y_O))$ are
fast oscillating periodic functions, of period $\frac{2\pi}{\xi}$,
while the Gaussian factor is almost constant at this scale, we
approximate the integral by splitting the space into hypercubes $H_k$
of side length $\frac{2\pi}{\xi}$.

Each $H_k$ has diameter $2\pi\frac{\sqrt N}{\xi}$.  Let us notice that
$\int g_n = 1$ and that for every $k\in\ma Z^N$ and
$y=(y_O)_{O\in\mathrm{Per}(n)}$,
\begin{equation}\label{periodicite de f_O} \int_{H_k}
\left(\prod_{O\in\mathrm{Per}(n)} f_O(\xi(x_O+y_O))\right)dx=\frac
1{\xi^N}\prod_{O\in\mathrm{Per}(n)}\int_0^{2\pi}f_O(t)\mathrm{d}t.
\end{equation}

For any fixed $y=(y_O)_{O\in\mathrm{Per}(n)}\in\ma
R^{\mathrm{Per}(n)}$, by periodicity,
\begin{equation}
    \label{inegalite triangulaire 1}
    \begin{split}
\left|E_y(\mu,\nu)-\prod_{O\in\mathrm{Per}(n)}\int_0^{2\pi}f_O(t)\mathrm
dt\right|=\left|E_y(\mu,\nu)-\left(\prod_{O\in\mathrm{Per}(n)}\int_0^{2\pi}f_O(t)\mathrm
dt\right)\int g_n\right|\\
\leq\left|E_y(\mu,\nu)-\left(\prod_{O\in\mathrm{Per}(n)}\int_0^{2\pi}f_O(t)\mathrm
dt\right)\left(\frac{2\pi}{\xi}\right)^N\sum_{k\in\ma
Z^{\mathrm{Per}(n)}}g_n\left(2\pi\frac k\xi\right)\right|\\
+\left|\prod_{O\in\mathrm{Per}(n)}\int_0^{2\pi}f_O(t)\mathrm
dt\right|\left|\left(\frac{2\pi}{\xi}\right)^N\sum_{k\in\ma
Z^{\mathrm{Per}(n)}}g_n\left(2\pi\frac k\xi\right)-\int
g_n\right|\end{split}
\end{equation}
Splitting the integral in (\ref{E_y}) into a sum of integrals over the
$H_k$, and using (\ref{periodicite de f_O}) allows us to bound the
first term of the right hand-side of (\ref{inegalite triangulaire 1})
by
\begin{multline} \left|\sum_{k\in\ma
Z^{\mathrm{Per}(n)}}\int_{H_k}\left(\prod_{O\in\mathrm{Per}(n)}f_O(\xi(x_O+y_O))\right)\left(g_n(x)-g_n\left(2\pi\frac
k\xi\right)\right)\mathrm{d}x\right|\\ \leq
\left|\prod_{O\in\mathrm{Per}(n)}\int_0^{2\pi}f_O(t)\mathrm
dt\right|2\pi\frac{\sqrt N}{\xi}\sum_{k\in\ma
Z^N}\frac{\sup_{H_k}\|\nabla g_n\|}{\xi^N}
\end{multline} by mean-value inequality.  Likewise, in the second
term,
\begin{equation}
\begin{split} \left|\left(\frac{2\pi}{\xi}\right)^N\sum_{k\in\ma
Z^{\mathrm{Per}(n)}}g_n\left(2\pi\frac k\xi\right)-\int
g_n\right|&=\left|\sum_{k\in\ma
Z^{\mathrm{Per}(n)}}\int_{H_k}\left(g_n\left(2\pi\frac k\xi\right)-
g_n(x)\right)\mathrm{d}x\right|\\ &\leq 2\pi\frac{\sqrt
N}{\xi}\sum_{k\in\ma
Z^{\mathrm{Per}(n)}}\left(\frac{2\pi}{\xi}\right)^N\sup_{H_k}\|\nabla
g_n\|.
\end{split}
\end{equation}
This ends the proof of Lemma \ref{lemme majoration E(mu,nu)}.
\end{proof}

\begin{proof}[Proof of Lemma \ref{lemme accroissements finis}]
\begin{equation}\label{expression norme gradient g_n} \|\nabla
  g_n(x)\|=\frac{\|{\Sigma_n^{-1}}x\|e^{-\frac12\|\sqrt{\Sigma_n^{-1}}x\|^2}}
  {\sqrt{(2\pi)^N\det(\Sigma_n)}}\leq\|\sqrt{\Sigma_n^{-1}}\|
  \frac{\|\sqrt{\Sigma_n^{-1}}x\|
    e^{-\frac12\|\sqrt{\Sigma_n^{-1}}x\|^2}} {\sqrt{(2\pi)^N\det(\Sigma_n)}}.
\end{equation}
Since $\|\nabla g_n\| $ is a function of
$\|\sqrt{\Sigma_n^{-1}}x\|$, we will pack the terms of the sum
\begin{equation}\label{eq sup norme gradient} \sum\limits_{k\in\ma
Z^N}\left(\frac{2\pi}{\xi}\right)^N\sup\limits_{H_k}\|\nabla g_n\|
\end{equation}
by level sets of $\|\Sigma_n^{-1}\|$.  Let us first
observe that, by (\ref{eq sigma n})
\begin{equation}
\sqrt{\Sigma_n^{-1}}=\mathrm{Diag}\left(\frac1{\sigma_O}\right)_{O\in\mathrm{Per}(n))}+O(e^{-kn})
\end{equation}
for every $k$ so, by Lemma \ref{lemme X^n_O Y^n_O}
\begin{equation}\label{norme sqrt C_n^-1}
    \begin{split} \|\sqrt{\Sigma_n^{-1}}\|&\underset{(\ref{matrice de
covariance C_n diag})}{=}O(h_n^{-d(\alpha-\frac12)-\gamma})\\
&\underset{(\ref{definition h_j})}{=}O(\widetilde\Lambda^{\frac
n2(d(\alpha-\frac12)+\gamma)}).
    \end{split}
\end{equation}

Provided that $\widetilde\Lambda>\Lambda$ and $\gamma>0$ are chosen
small enough with respect to $\frac1c$ in (\ref{relation n xi}),
writing $1+\delta:=\frac 1c$, we have
\begin{equation}
\begin{split} \xi&\hspace{0.45cm}\underset{(\ref{relation n
xi})}{\geq}\hspace{0.45cm}e^{\frac nc h_{\mathrm{top}}}\Lambda^{\frac
{nd}{2c}(\alpha-\frac 12)}\\ &\underset{\phantom{(\ref{relation n
xi}),~(\ref{card per(n)})}}{\geq} e^{(1+\delta)nh_{\mathrm{top}}}
\widetilde\Lambda^{\frac n2(d(\alpha-\frac12)+\gamma)}\\
&\underset{(\ref{norme sqrt C_n^-1}),~(\ref{card
per(n)})}{\geq}CN^{1+\delta}\|\sqrt{\Sigma_n^{-1}}\|.
\end{split}
\end{equation}

So
\begin{equation}
    \label{relation norme C_n^-1 et sqrt N /xi} \exists \delta>0,
\|\sqrt{\Sigma_n^{-1}}\|=o\left(\frac{\xi}{N^{1+\delta}}\right).
\end{equation}
Let us now fix such a $\delta$ and write for $j\geq 0$

\begin{equation} C_j:=\left\{x\in\ma
R^N,\frac{j}{N^{\frac{1+\delta}{2}}}<\|\sqrt{\Sigma_n^{-1}}x\|\leq
\frac{j+1}{N^{\frac{1+\delta}{2}}}\right\}.
\end{equation}
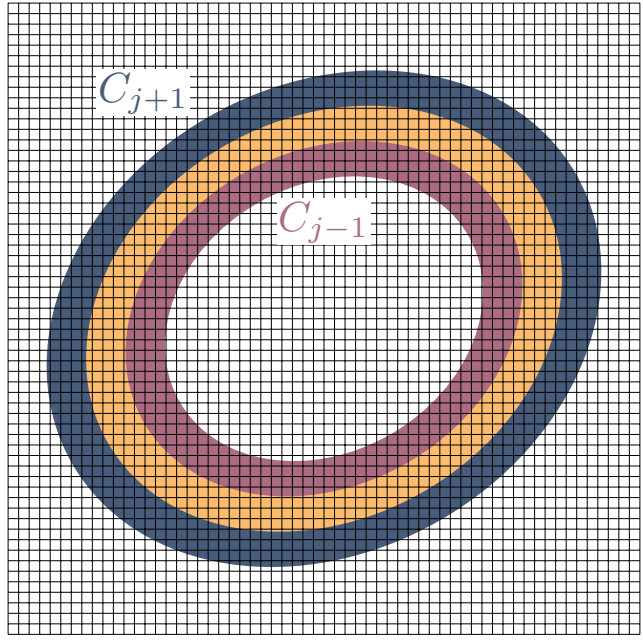
\begin{figure}[h!]\label{fig_entropie_partition}
  \begin{tikzpicture}[scale=0.2]
    \filldraw[thick,cm={cos(60),-sin(60),sin(60),cos(60),(0,0)}, noir,
    scale=1.1] (0,0) ellipse (14cm and 17.5cm);
    \filldraw[thick,cm={cos(60),-sin(60),sin(60),cos(60),(0,0)},
    jaune, scale=1.1] (0,0) ellipse (12cm and 15cm);
    \filldraw[thick,cm={cos(60),-sin(60),sin(60),cos(60),(0,0)},
    rouge, scale=1.1] (0,0) ellipse (10cm and 12.5cm);
    \filldraw[thick,cm={cos(60),-sin(60),sin(60),cos(60),(0,0)},
    White, scale=1.1] (0,0) ellipse (8cm and 10cm);
    \draw[very thin,scale=0.70] (-30,-30) grid (30,30);
    \fill[White] (-3.1,4.9) rectangle (3.1,8);
    \fill[White] (-15.1,13.5) rectangle (-8.9,16.6);
    \node[very thick,rouge, scale
    = 1.6] at (0,6.4) {$C_{j-1}$};
    \node[very thick,noir, scale = 1.6] at
    (-12,15) {$C_{j+1}$};
  \end{tikzpicture}
  \caption{The hypercubes $H_k$, pictured as a grid, have small
    diameter, relatively to the distance between the annuli $C_j,$
    $C_{j+2}$ (or $C_j$ and $C_{j-2}$).  Consequently, the supremum $\sup_{H_k}\|\nabla g_n\|$
    in (\ref{eq sup norme gradient}) can be replaced by a supremum
    over $C_{j-1}\cap C_j\cap C_{j+1}$ if $H_k$ intersercts $C_j$.} 
  \label{fig C_j}
\end{figure}

The diameter of the cubes $H_k$ is then very small compared to
the distance between $C_j$ and $C_{j+2}$ for every $j$ (see
Figure \ref{fig C_j}): Indeed,

\begin{equation} x\in C_{j}\implies \|\sqrt{\Sigma_n^{-1}}x\|\leq
\frac{j+1}{N^{\frac{1+\delta}{2}}}
\end{equation} and
\begin{equation} y\in C_{j+2}\implies \|\sqrt{\Sigma_n^{-1}}y\|\geq
\frac{j+2}{N^{\frac{1+\delta}{2}}}.
\end{equation}

So by triangular inequality
\begin{equation} \frac1{N^{\frac{1+\delta}{2}}}\leq
\|\sqrt{\Sigma_n^{-1}}y\|-\|\sqrt{\Sigma_n^{-1}}x\|\leq
\|\sqrt{\Sigma_n^{-1}}(x-y)\|\leq \|\sqrt{\Sigma_n^{-1}}\|\|x-y\|.
\end{equation}

Thus by (\ref{relation norme C_n^-1 et sqrt N /xi}), for all $j,k$,
\begin{equation}
d(C_{j-2},C_j)\geq\frac{1}{N^{\frac{1+\delta}{2}}\|\sqrt{\Sigma_n^{-1}}x\|}\geq
C\frac{N^{\frac{1+\delta}{2}}}{\xi}\gg 2\pi\frac{\sqrt
N}{\xi}=\mathrm{Diam}(H_k).
\end{equation}

Consequently, we have
\begin{equation}\label{inclusion} \{H_k,H_k\cap
C_j\neq\emptyset\}\subset C_{j-1}\cup C_j\cup C_{j+1}.
\end{equation}

Thus,
\begin{equation}\label{1}
    \begin{split} \sum_{k\in\ma
Z^N}\left(\frac{2\pi}{\xi}\right)^N\sup_{H_k}\|\nabla
g_n\|&\leq\sum_{j\in\ma N}\sum_{k,H_k\cap
C_j\neq\emptyset}\left(\frac{2\pi}{\xi}\right)^N\sup_{H_k}\|\nabla
g_n\|\\ &\leq\sum_{j\in\ma N}\#\{k,H_k\cap
C_j\neq\emptyset\}\left(\frac{2\pi}{\xi}\right)^N{\sup\limits_{C_{j-1}\cup
C_j\cup C_{j+1}}\|\nabla g_n\|}.
    \end{split}
\end{equation}

We deduce from (\ref{inclusion}) that
\begin{equation}\label{2}
    \begin{split} \#\{k,H_k\cap
C_j\neq\emptyset\}\left(\frac{2\pi}{\xi}\right)^N&=\#\{k,H_k\cap
C_j\neq\emptyset\}\cdot\mathrm{Vol}(H_k)\\
&\leq \mathrm{Vol}\left(\bigcup_{l\leq j+1}C_l \right)\\
&=\mathrm{Vol}\{x,\|\sqrt{\Sigma_n^{-1}}x\|\leq
\frac{j+2}{N^{\frac{1+\delta}{2}}}\}\\ &=\sqrt{\det
\Sigma_n}\left(\frac{j+2}{N^{\frac{1+\delta}{2}}}\right)^N\mathrm{Vol}(B(0,1)).
\end{split}
\end{equation}
And we also have
\begin{equation}\label{3} \sup_{C_{j-1}\cup C_j\cup C_{j+1}}\|\nabla
g_n\|\underset{(\ref{relation norme C_n^-1 et sqrt N
/xi}),~(\ref{expression norme gradient g_n})}{\leq}
\frac{1}{N^{\frac12+\delta}}\frac{\xi}{\sqrt
N}\frac{j+2}{N^{\frac{1+\delta}{2}}}\frac{e^{-\frac12\frac{(j-1)^2}{N^{1+\delta}}}}{\sqrt{(2\pi)^N\det(\Sigma_n)}}
\end{equation}
We know moreover (see for instance \cite{ball1997elementary} p.5) that
the unit ball of dimension $N$ has a volume equivalent to
\begin{equation}\label{4} \frac{1}{\sqrt{\pi}}\frac{(2\pi
e)^{N/2}}{N^{\frac{N+1}2}}.
\end{equation}

Finally, putting together (\ref1), (\ref2), (\ref3),
and (\ref4), we have the following upperbound for the sum
\begin{equation}\label{majoration du majorant}
  \begin{split} \frac{\sqrt N}{\xi}\sum_{k\in\ma
      Z^N}&\left(\frac{2\pi}{\xi}\right)^N\sup_{H_k}\|\nabla g_n\|\\
    &\underset{(\ref 1),\ (\ref 2)}{\leq}\frac{\sqrt N}{\xi}
    \mathrm{Vol}B(0,1)\sqrt{\det\Sigma_n}\sum_{j\geq0}
    \left(\frac{j+2}{N^{\frac{1+\delta}{2}}}\right)^N
    \sup_{C_{j-1} \cup C_j\cup
      C_{j+1}}\|\nabla g_n\|\\
    &\hspace{0.43cm}\underset{(\ref 3)}{\leq}\hspace{0.4cm}\frac{1}{N^{\frac12+\delta}} \mathrm{Vol}B(0,1)
    \sum_{j\geq0}\left(\frac{j+2}{N^{\frac{1+\delta}{2}}}\right)^{N+1}
    \frac{e^{-\frac12\frac{(j-1)^2}{N^{1+\delta}}}}{(2\pi)^{N/2}}\\
    &\hspace{0.43cm}\underset{(\ref 4)}{\leq}\hspace{0.4cm}C\frac{e^{N/2}}{N^{(N+1)\frac{1+\delta}2}}\frac{1}
    {N^{\frac12+\delta}}\frac{1}
    {N^{\frac{N+1}2}}\sum_{j\geq0}(j+2)^{N+1}
    e^{-\frac12\frac{(j-1)^2}{N^{1+\delta}}}\\
    &\hspace{0.67cm}=\hspace{0.64cm} C\frac{e^{N/2}}{N^{(N+1)(1+\frac \delta 2)}}\frac{1}
    {N^{\frac12+\delta}}\sum_{j\geq0}(j+2)^{N+1}
    e^{-\frac12\frac{(j-1)^2}{N^{1+\delta}}}.
  \end{split}
\end{equation}
To conclude, we will now show that

\begin{equation} \frac{
e^{N/2}}{N^{(N+1)(1+\frac\delta2)}}\sum_{j\geq0}(j+2)^{N+1}e^{-\frac12(j-1)^2}=
O\left(N^{\frac12+\frac{2\delta}{3}}\right):
\end{equation}
\begin{lemma}  
  Let
  \begin{equation}
    \fonction{f_N}{\ma R}{\ma R}{x}{\frac{
e^{N/2}}{N^{(N+1)(1+\frac\delta2)}}(x+2)^{N+1}e^{-\frac{(x-1)^2}{2N^{1+\delta}}}}.
\end{equation}
It is first increasing then decreasing and admits a maximum $O(1)$ at
\begin{equation}
  x_0\sim N^{1+\frac\delta2}
\end{equation}
and decays at a scale $N^{\frac12+\frac23\delta}$:
If we write for $k\geq-1$,
\begin{equation}
x_k=x_0+kN^{\frac12+\frac23\delta},
\end{equation}
we have
\begin{equation} f_N(x_k)=O(e^{k-\frac12k^2N^{\frac\delta3}})
\end{equation}
as $N$ goes to infinity.
\end{lemma}
\begin{proof}
Differentiating the logarithm of $f_N$ tells us that the function is
increasing, then decreasing, and that the maximum is attained for $x_0$ such
that
\begin{equation} x_0^2+x_0-2=N^{2+\delta}+N^{1+\delta}.
\end{equation}

This shows that
\begin{equation}
  x_0=N^{1+\frac\delta2}+\frac12N^{\frac\delta2}(1+o(1)),
\end{equation}
so 
\begin{equation}
x_0=N^{1+\frac\delta2}+O(N^{\frac\delta2})=N^{1+\frac\delta2}(1+O(N^{-1}))
\end{equation}
and hence
\begin{multline} \max\log f_N=\frac N2-(N+1)(1+\frac\delta2)\log N\\
+(N+1)\log\left(N^{1+\frac\delta2}(1+O(N^{-1}))\right)-\frac12\frac{\left(N^{1+\frac\delta2}+O(N^{\frac\delta2})\right)^2}{N^{1+\delta}}=O(1)
\end{multline}

\begin{equation}
x_k=x_0+kN^{\frac12+\frac23\delta}=N^{1+\frac\delta2}+kN^{\frac12+\frac23\delta}+O(N^{\frac\delta2})=N^{1+\frac\delta2}(1+kN^{-\frac12+\frac\delta6}+O(N^{-1})).
\end{equation}
The remainder does not depend on $k.$
Therefore, for $N\geq1$
\begin{multline} \log f_N(x_k)=\frac N2- (N+1)(1+\frac\delta2)\log N\\
+(N+1)\log\left(N^{1+\frac\delta2}(1+kN^{-\frac12+\frac\delta6}+O(N^{-1}))\right)-\frac12\frac{\left(N^{1+\frac\delta2}+kN^{\frac12+\frac23\delta}+O(N^{\frac\delta2})\right)^2}{N^{1+\delta}}\\
\leq
k(N+1)N^{-\frac12+\frac\delta6}+O(1)-kN^{\frac12+\frac\delta6}-\frac12k^2N^{\frac\delta3}+O(1)\\
=-\frac12k^2N^{\frac\delta3}+k+O(1).
  \end{multline}

  Consequently,
\begin{equation}
  \begin{split}
\sum_{j\geq0}f_N(j)&=\sum_{j<x_{-1}}f_N(j)+\sum_{x_{-1}\leq
j<x_1}f_N(j)+\sum_{k=1}^{+\infty}\sum_{x_k\leq j<x_{k+1}}f_N(j)\\
&\leq
f_N(x_{-1})O(N^{1+\delta})+O(N^{\frac12+\frac23\delta})+O(N^{\frac12+\frac23\delta})\\
&=O(N^{\frac12+\frac23\delta}).
  \end{split}
\end{equation}
\end{proof}
Together with (\ref{majoration du majorant}), this concludes the proof
of Lemma \ref{lemme accroissements finis}, and hence the proof of
Proposition \ref{prop emunu}.
\end{proof}
\end{proof}

In order to conclude the proof, we need the following lemma, whose
proof is postponed to Appendix \ref{annexe pression}.
\begin{lemma}\label{lemme pression}

\begin{equation}
nA_n\sup_{O\in\mathrm{Per}(n)}\frac{1}{|\det(1-dT^n_O)|}\underset{{n\to\infty}}{\longrightarrow}0.
\end{equation}

\end{lemma}

Since
\begin{equation} \int_{\ma R}e^{i(\mu\cos u+\nu\sin u)}\mathrm du =
1-\frac{\mu^2+\nu^2}{4}+o(\mu^2+\nu^2),
\end{equation}
we obtain for a given $(\mu,\nu)\in\ma R^2$

\begin{equation}
\begin{split} \prod_{O\in\mathrm{Per}(n)}\int_{\ma R}f_O(u)\mathrm du
&=\prod_{m|n}\prod_{O\in\mc
P_m}\left(1-\frac{\mu^2+\nu^2}{4}\frac{m^2A_n^2}{|\det(1-dT^n_O)|^2}+o\left(\frac{m^2A_n^2}{|\det(1-dT^n_O)|^2}\right)\right)\\
&=e^{\sum\limits_{m|n}\sum\limits_{O\in\mc
P_m}\log\left(1-\frac{\mu^2+\nu^2}{4}\frac{m^2A_n^2}{|\det(1-dT^n_O)|^2}+o\left(\frac{m^2A_n^2}{|\det(1-dT^n_O)|^2}\right)\right)}\\
&=e^{-\frac{\mu^2+\nu^2}{4}\sum\limits_{m|n}\sum\limits_{O\in\mc
P_m}\frac{m^2A_n^2}{|\det(1-dT^n_O)|^2}+o\left(\sum\limits_{m|n}\sum\limits_{O\in\mc
P_m}\frac{m^2A_n^2}{|\det(1-dT^n_O)|^2}\right)}\\
&\underset{(\ref{A_n})}{=}e^{-\frac{\mu^2+\nu^2}{4}+o(1)}.
\end{split}
\end{equation}

Levy's theorem then implies the convergence in law of $A_n
\mathrm{Tr}^\flat(\mc L^n_{\xi,\tau})$ towards a complex Gaussian law
$\mc N_{\ma C}(0,1)$.

\appendix
\section{Upper bound for the distance between periodic points of a
given period}\label{appendix} Note that for a given period, there is a
finite number of periodic points, whose mutual distance is bounded
below: Then,
\begin{lemma}\label{lemme distance points periodiques}
  \begin{equation} \forall\Lambda'>\Lambda,\exists C>0,\forall n\in\ma
N,\forall x\neq y\in M, (T^nx=x, T^ny=y)\implies d(x,y)\geq \frac
C{\Lambda'^{\frac n2}}.
   \end{equation}
\end{lemma}

Anosov diffeomorphisms are expansive (see
\cite{katok1997introduction}, p.125):

\begin{equation} \exists \delta>0,\forall x,y\in M,\left(\forall
k\in\ma Z,d(T^k(x),T^k(y))\leq \delta\right)\implies x=y.
\end{equation}

Let $\Lambda'>\Lambda$, and let $n_0\in\ma N$ such that
\begin{equation} \forall |n|\geq n_0,\max_{x\in M}\|dT^n_x\|\leq
\Lambda'^{|n|}.
\end{equation}

Let $L = \max_{x\in M}\|dT^{\pm1}_x\|\geq\Lambda.$ For any $x,y\in M$,

\begin{equation} d(T^kx,T^ky)\leq\left\lbrace\begin{array}l
\Lambda'^{|k|}\text{ if }|k|\geq n_0 \\ L^{|k|} \text{ if } |k|<n_0
\end{array}\right.
\end{equation}

Writing $C=\left(\frac{L}{\Lambda'}\right)^{n_0}$, we get
for all $n\in\ma N$, for all periodic points
$x,y$ of period $n$, for all $-n/2\leq k\leq n/2,$

\begin{equation} d(T^kx,T^ky)\leq C\Lambda'^{n/2}d(x,y).
\end{equation}

Therefore, if $d(x,y)\leq\frac{\delta}{C\Lambda'^{n/2}},$ then for all
$-\frac n2\leq k\leq \frac n2,$ (thus for all $k\in\ma Z$)
$d(T^kx,T^ky)\leq \delta$ and $x=y$.

\section{Proof of Lemma \ref{lemme pression}}
\label{annexe pression}
\subsection{Topological pressure}
\begin{proposition}[{\cite[Proposiotion 20.3.3]{katok1997introduction}}]
  Let $\phi:M\longrightarrow \ma R$ be a Hölder function.
  The topological pressure of $\phi$ is given by
  \begin{equation}
    \mathrm{Pr}(\phi)=\lim_{n\to\infty}\frac1n\log\left(\sum_{x=T^n(x)}e^{\phi^n_x}\right).
  \end{equation}
\end{proposition}
\subsection{Variationnal principle}
The topological pressure of a Hölder function is given by a
variationnal principle:
\begin{theorem}[{\cite[Theorem 2.17]{bowen1975ergodic}}]
  Let $\phi:M\longrightarrow M$ be a Hölder function.
  \begin{equation}
    \mathrm{Pr}(\phi)=\sup_\mu(h_\mu(T)+\int_M \phi\ \mathrm d\mu)
  \end{equation}
  where $\mu$ goes through the $T$-invariant probability measures. 
  This supremum is attained at a unique measure \cite[Propostion
  20.3.7]{katok1997introduction} called equilibrium measure of $\phi$.
\end{theorem}
Let us write
\begin{equation}
  J_u(x):=\log|\det\mathrm dT_{|E^u(x)}|.
\end{equation}
 $J_u$ is said to be cohomologous to a constant if 
  \begin{equation}
    \exists c\in\ma R,\exists h\in\mc C(M),\ J_u=c+h\circ T-h.
 \end{equation}
 
\begin{lemma}[{\cite[Proposition 20.3.10]{katok1997introduction}}]
  Let $\mu_\beta$ be the equilibrium measure of $-\beta J_u$, $\beta>0$.
  If $J_u$ is not cohomologous to a constant, the map $\beta\in\ma R_+\mapsto \mu_\beta$ is injective.  
\end{lemma}
\begin{corollary}
  As a consequence, the function
  \begin{equation}
    \fonction F {\ma R_+^*}{\ma R}\beta{\frac 1\beta\mathrm{Pr}(-\beta J_u)}
  \end{equation}
  is strictly decreasing.
\end{corollary}
\begin{proof}
  If $J_u$ is cohomologous to a constant $c$, then
  by Proposition \ref{corollaire card per(n)},
  \begin{equation}
    \sum_{T^n(x)=x}e^{-\beta (J_u)^n_x}\asymp e^{n(h_{\mathrm{top}}-c\beta)}.
  \end{equation}
  Hence,
  \begin{equation}
    \mathrm{Pr}(-\beta J_u)=h_{\mathrm{top}}-c\beta,
  \end{equation}
  with $h_{\mathrm{top}}>0$, which gives the wanted result.  If $J_u$ is
  not cohomologous to a constant, then the previous Lemma applies.
  Let $\beta'>\beta>0$.
  \begin{equation}
    \int -\beta J_u\ \mathrm d\mu_\beta=\mathrm{Pr}(-\beta J_u)>\int -\beta' J_u\ \mathrm d\mu_{\beta'}
  \end{equation}
  
  Consequently,
  \begin{equation}
    F(\beta)=\int -J_u\ \mathrm d\beta  +\frac{h(\mu_\beta)}{\beta}>\int
    -J_u\ \mathrm d\mu_{\beta'}+\frac{h(\mu_{\beta'})}{\beta}\geq\int
    -J_u\ \mathrm d\mu_{\beta'}+\frac{h(\mu_{\beta'})}{\beta'}=F(\beta').
 \end{equation}
\end{proof}
$F$ is moreover bounded below, and admits consequently a limit as
$\beta$ goes to infinity.
\subsection{End of proof of Lemma \ref{lemme pression}}
The quantity $\inf_{x\in M}(J_u)^n_x$ is subadditive.  We can
consequently define
\begin{equation}
  J_u^{\min}:=\lim_{n\to\infty}\inf_{x\in M}\frac1n(J_u)^n_x.
\end{equation}
\begin{lemma}
  The infimum can be taken over periodic points of period $n$:
  \begin{equation}
    J_u^{\min} =\lim_{n\to\infty}\inf_{T^n(x)=x}\frac1n(J_u)^n_x.
  \end{equation}
\end{lemma}
\begin{proof}
  For $n\in\ma N$, let $x_n\in M$ be the point such that
  \begin{equation}
    (J_u)^n_{x_n}=\inf_{x\in M}(J_u)^n_x.
  \end{equation}
  Let $\varepsilon>0$.  By the specification property \cite[Theorem
  18.3.9]{katok1997introduction}, there exists $M\in\ma N$ and a
  periodic point $x_n'$ of period $n+M$ such that
  \begin{equation}
    \forall\ 0\leq j\leq n-1,\ d(T^j(x_n),T^j(x_n'))\leq\varepsilon.
  \end{equation}
  Thus, for some $C>0$,
  \begin{equation}
    \frac1n\inf_{x\in
      M}(J_u)^n_x\leq\frac1n\inf_{T^{n+M}(x)=x}(J_u)^n_x\leq\frac1n\inf_{x\in
      M}(J_u)^n_x+C\varepsilon.
  \end{equation}
  Then, we have
  \begin{equation}
    \begin{split}
      \frac1{n+M}\inf_{x\in
        M}(J_u)^{n+M}_x-C\frac{M}{n}&\leq\frac1{n+M}\inf_{T^{n+M}(x)=x}(J_u)^{n+M}_x\\
      &\leq\frac1{n+M}\inf_{x\in
        M}(J_u)^{n+M}_x+C\varepsilon+C\frac{M}{n}.
    \end{split}
  \end{equation}
  As a consequence,
  \begin{equation}
    \inf_{T^n(x)=x}\frac1n(J_u)^n_x\underset{n\to\infty}{\longrightarrow}J_u^{\min}.
  \end{equation}
\end{proof}
\begin{lemma}
  \begin{equation}
    F(\beta)\underset{\beta\to\infty}{\longrightarrow}-J_u^{\min}
  \end{equation}
\end{lemma}
\begin{proof}
Let 
  \begin{equation}
    F_n(\beta):=\frac1{n\beta}\log\sum_{T^n(x)=x}e^{-\beta (J_u)^n_x}\underset{n\to\infty}{\longrightarrow}F(\beta).
  \end{equation}
  Let $\varepsilon>0$.
  If $n$ is large enough, for every periodic point $x$ of period $n$
  \begin{equation}
    (J_u)^n_x\geq n(J_u^{\min}-\varepsilon).
  \end{equation}
  Moreover, there exists a periodic point of period $n$ such that
  \begin{equation}
    (J_u)^n_x\leq n(J^{\min}+\varepsilon).
  \end{equation}
  So
  \begin{equation}
    e^{-\beta n(J_u^{\min}+\varepsilon)}\leq\sum_{T^n(x)=x}e^{-\beta
      (J_u)^n_x}\leq e^{nh_{\mathrm{top}}}e^{-\beta n(J_u^{\min}+\varepsilon)}.
  \end{equation}
  Taking the logarithm and taking $\varepsilon\to0$ gives for every
  $\beta>0$
  \begin{equation}
    -J_u^{\min}\leq F(\beta)\leq\frac{h_{\mathrm{top}}}\beta-J_u^{\min},
  \end{equation}
  hence the result.
\end{proof}
The proof of Lemma \ref{lemme pression} then derives from the
following facts:
\begin{equation}\label{sec:end-proof-lemma}
  \sup_{O\in\mathrm{Per}(n)}\frac1{|\det{(1-\mathrm dT^n_O)}|}=e^{-nJ_u^{\min}}(1+o(1))=e^{n\lim_{\beta\to\infty}F(\beta)+o(n)}
\end{equation}
and
\begin{equation}\label{eq:fin}
  A_n=e^{nF(2)+o(n)}.
\end{equation}
(\ref{sec:end-proof-lemma}) is clear,
and the definition (\ref{A_n}) implies that
\begin{equation}
  \frac1{\sqrt n}\left(\sum_{T^n(x)=x}e^{-n(J_u)^n_x}(1+o(1))\right)^{-\frac12}\leq A_n\leq\left(\sum_{T^n(x)=x}e^{-n(J_u)^n_x}(1+o(1))\right)^{-\frac12}.
\end{equation}

\section{Proof of Proposition \ref{proposition regularite champs gaussiens}}

\label{annexe champs gaussiens}

These results are based on the fact that if $\zeta_j$ are i.i.d.
centered Gaussian variables of variance $1$, then, almost surely, for
every $\varepsilon>0,$
\begin{equation}\label{croissance famille gaussienne}
\zeta_j=o(j^\varepsilon).
\end{equation}
It is a consequence of Borel-Cantelli Lemma:
\begin{equation}
  \ma P\left(|\zeta_j|>j^\varepsilon\right)\leq\frac2{\sqrt{2\pi}}e^{-j^{2\varepsilon}/2},
\end{equation}
it is thus summable.
The second ingredient is Weyl's law (\ref{eq:Weyl law}): the
asymptotics of the sequence of eigenvalues $(\lambda_j)$ of the
Laplace-Beltrami opertor is given by
\begin{equation}\label{Loi de
Weyl} \lambda_j\sim Cj^{\frac{2}{d}}
\end{equation}
for some constant $C$ depending only on $M$ (and the
metric).

\begin{proof}[Proof of Proposition \ref{proposition regularite champs
    gaussiens}]
  If $c_j=O(j^{-\alpha})$, then for $s<d(\alpha-\frac12)$,
or equivalentely $\alpha>\frac sd+\frac 12$, almost surely,

\begin{equation} |c_j\zeta_j|^2\underset{(\ref{croissance famille
gaussienne})}{=}O(j^{-2s/d-1-\delta})
\end{equation} for some $\delta>0.$ Thus,

\begin{equation} \sum_j|c_j\zeta_j|^2j^{2\frac s d}
\end{equation}
converges, and so does
\begin{equation} \sum_j |c_j\zeta_j|^2 (1+\lambda_j)^s
\end{equation}
by Weyls law (\ref{Loi de Weyl}).  Consequently, almost surely,
\begin{equation} (1+\Delta)^{s/2} f=\sum_j c_j \zeta_j
(1+\lambda_j)^{\frac{s}{2}} \phi_j\in L^2.
\end{equation}

In other terms, $f$ belongs almost surely to $H^s$.  The statement
about $\mc C^k(M)$ follows from the classical Sobolev embedding
theorem:
\begin{theorem}[{\cite[Proposition 3.3
p.282]{taylor1996partial}}]\label{lemme continuite champs gaussiens}
For $s>k+d/2$,
\begin{equation} H^s(M)\subset \mc C^k(M).
\end{equation}
\end{theorem}
\end{proof}
\section{Ruelle spectrum and flat trace}
\label{annexe spectre de Ruelle}

\begin{theorem}[{\cite[Theorem 5.1]{baladi2018dynamical}}] In our
  setting, if $\tau$ is $\mc C^k$, there
  exists a Banach space $B^k$ such that
    \begin{itemize}
    \item $\mathcal L_{\xi,\tau}:B^k\longrightarrow B^k$ is a bounded
      operator,
    \item For any $\alpha>\frac k2$,
      \begin{equation}\label{espace B^s} \mathcal
        C^\alpha(M)\subset B^k\subset\mathcal C^\alpha(M)',
      \end{equation}
    \item The spectral radius $r(\mathcal L_{\xi,\tau})$ of $\mathcal
      L_{\xi,\tau}$ in $B^k$ is
      \begin{equation}\label{} r(\mathcal L_{\xi,\tau})\leq
        1,
      \end{equation}
    \item The essential spectral radius $r_{ess}(\mathcal L_{\xi,\tau})$
      of $\mathcal L_{\xi,\tau}$ in $B^k$ satisfies
      \begin{equation}\label{borne r_ess}
        r_{ess}(\mathcal L_{\xi,\tau})\leq m^k,
      \end{equation}
      where $\lambda$ is the smallest constant satisfying the
      definition of Anosov diffemorphism from Section \ref{sec:model}.
    \end{itemize}
  \end{theorem}
This implies that the resolvent of the transfer operator admits a
meromorphic extension as an operator $\mathcal
C^\infty(M)\longrightarrow\mc D'(M)$ to the region
\begin{equation}
  \{|z|>m^k\}
\end{equation}
and that its poles in this domain are the eigenvalues of
$\mc L_{\xi,\tau}$ with same multiplicities and spectral
projectors.  These poles are called Ruelle-Pollicott resonances.  For
analytic Anosov maps, there are spaces in which the transfer operator
is trace-class.  There is however no hope for this to be true with
less regularity.  Yet, an analog of the trace can be defined, that
coincides with the usual trace in the analytic case:
\begin{lemma}[{\cite[Section 6.2]{baladi2018dynamical}}]
  The operators $\mc L^n_{\xi,\tau},n\geq 1$ have integral kernel

\begin{equation} K_{\xi,\tau}^n(x,y)=e^{i\xi\tau_x^n}\delta(y-T^n(x)),
\end{equation}
where
\begin{equation} \tau^n_x:=\sum_{j=0}^{n-1}\tau(T^j(x)).
\end{equation}
Introducing a mollification, the distribution $K^n_{\xi,\tau}$ can be
integrated along the diagonal $\{(x,x),x\in M\}$.  The resulting
quantity is called flat trace of $\mc L^n_{\xi,\tau} $ and can be
expressed as a sum over periodic points:
\begin{equation}\label{trace}\mathrm{Tr}^\flat\mc L^n_{\xi,\tau}:=
\int_{M} K^n_{\xi,\tau}(x,x)dx =
\sum_{x,T^n(x)=x}\frac{e^{i\xi\tau_x^n}}{\left|\det(1-d(T^n)_x)\right|}.
\end{equation}
\end{lemma}
This notion of trace is linked to the eigenvalues of the operator in
the following way (\cite[Theorem 6.2]{baladi2018dynamical},
\cite[Theorem 2.4]{jezequel2017local}):
\begin{theorem} Let $\xi\in\ma R$,
let $\varepsilon>0$ be such that $\mc L_{\xi,\tau}$ has no resonance
of modulus $\varepsilon$, then
\begin{equation}
\label{baladi_eq} \exists C_\xi>0,\forall n\in \ma
N,\left|\mathrm{Tr}^\flat\mc L^n_{\xi,\tau} -
\sum\limits_{\substack{\lambda\in\mathrm{Res}(\mc
L_{\xi,\tau})\\\abs{\lambda}>\varepsilon}}\lambda^n \right|\leq
C_\xi\varepsilon^n.
\end{equation}
\end{theorem}

\bibliographystyle{amsalpha}

\newpage \bibliography{main}

\end{document}